\documentclass[10pt]{amsart}  
\usepackage{amsmath,amssymb,amsfonts,graphicx, hyperref}

\title{Fourier analysis, linear programming, and densities of distance
  avoiding sets in $\R^n$}

\author{Fernando M\'ario de Oliveira Filho}
\address{F.M.~de Oliveira Filho, Centrum voor Wiskunde en Informatica (CWI),
Kruislaan 413, 1098 SJ Amsterdam, The Netherlands}
\email{f.m.de.oliveira.filho@cwi.nl}

\author{Frank Vallentin} 
\address{F.~Vallentin, Centrum voor Wiskunde en Informatica (CWI),
Kruislaan 413, 1098 SJ Amsterdam, The Netherlands}
\email{f.vallentin@cwi.nl}

\thanks{The first author was partially supported by CAPES/Brazil under
  grant BEX 2421/04-6. The second author was partially supported by
  the Deutsche Forschungsgemeinschaft (DFG) under grant SCHU
  1503/4. The second author thanks the Hausdorff Research Institute
  for Mathematics (Bonn) for its hospitality and support.}

\subjclass{42B05, 52C10, 52C17, 90C05} 
\keywords{Nelson-Hadwiger problem, measurable chromatic number, linear programming, almost periodic functions, autocorrelation function}

\date{November 19, 2008}

\newcommand{\defi}[1]{\textit{#1}}

\newcommand{\R}{\mathbb{R}}

\newcommand{\Z}{\mathbb{Z}}
\newcommand{\C}{\mathbb{C}}

\newtheorem{defin}{Definition}[section]

\newtheorem{theorem}[defin]{Theorem}

\DeclareMathOperator{\ort}{O}
\DeclareMathOperator{\vol}{vol}

\newcommand{\chim}{\chi_{\mathrm{m}}}

\newcommand{\odelta}{\overline{\delta}}

\begin{document}

\begin{abstract}
  In this paper we derive new upper bounds for the densities of
  measurable sets in $\R^n$ which avoid a finite set of prescribed
  distances. The new bounds come from the solution of a linear
  programming problem.  We apply this method to obtain new upper
  bounds for measurable sets which avoid the unit distance in
  dimensions $2, \ldots, 24$. This gives new lower bounds for the
  measurable chromatic number in dimensions $3, \ldots, 24$.  We apply
  it to get a short proof of a variant of a recent result of Bukh
  which in turn generalizes theorems of Furstenberg, Katznelson, and
  Weiss and Bourgain and Falconer about sets avoiding many distances.
\end{abstract}

\maketitle

\markboth{F.M.~de Oliveira Filho, F.~Vallentin}{Fourier analysis,
  linear programming, and densities of distance avoiding sets}

\section{Introduction}
\label{sec:introduction}

Let $d_1$, \dots,~$d_N$ be positive real numbers.  We say that a
subset~$A$ of the $n$-dimensional Euclidean space~$\R^n$ \defi{avoids
  the distances} $d_1$, \dots,~$d_N$ if the distance between any two
points in~$A$ is never $d_1$, \dots,~$d_N$. We define the \defi{upper
  density} of a Lebesgue measurable set $A \subseteq \R^n$ as
\begin{equation*}
  \overline{\delta}(A) = \limsup_{T \to \infty} \frac{\vol(A \cap [-T,T]^n)}{\vol [-T,T]^n}.
\end{equation*}
In this expression $[-T,T]^n$ denotes the regular cube in~$\R^n$ with
side~$2T$ centered at the origin. We denote the \defi{extreme density} which
a measurable set in~$\R^n$ that avoids the distances $d_1$, \dots,~$d_N$ 
can have by
\begin{equation*}
\begin{split}
m_{d_1, \ldots, d_N}(\R^n) = \sup\{\,\overline{\delta}(A) \; : \; \hbox{}& \text{$A \subseteq \R^n$ is measurable}\\
& \qquad \text{and avoids distances $d_1$, \dots,~$d_N$}\,\}.
\end{split}
\end{equation*}

In this paper we derive upper bounds for this extreme density from the
solution of a linear programming problem. 

To formulate our main theorem we consider the function~$\Omega_n$
given by
\begin{equation}
\label{eq:omegan}
\Omega_n(t) = \Gamma\Big(\frac{n}{2}\Big) \Big(\frac{2}{t}\Big)^{\frac{1}{2}(n-2)} J_{\frac{1}{2}(n-2)}(t),\;\;\text{for $t > 0$},\quad \Omega_n(0) = 1,
\end{equation}
where $J_{\frac{1}{2}(n-2)}$ is the \defi{Bessel function of the first
  kind} with \defi{parameter} $(n-2)/2$. To fix ideas we plotted the
graph of the function $\Omega_4$ in Figure~\ref{fig:omega}.

\renewcommand{\thefigure}{\arabic{section}.\arabic{figure}}

\begin{figure}
\begin{center}
\includegraphics{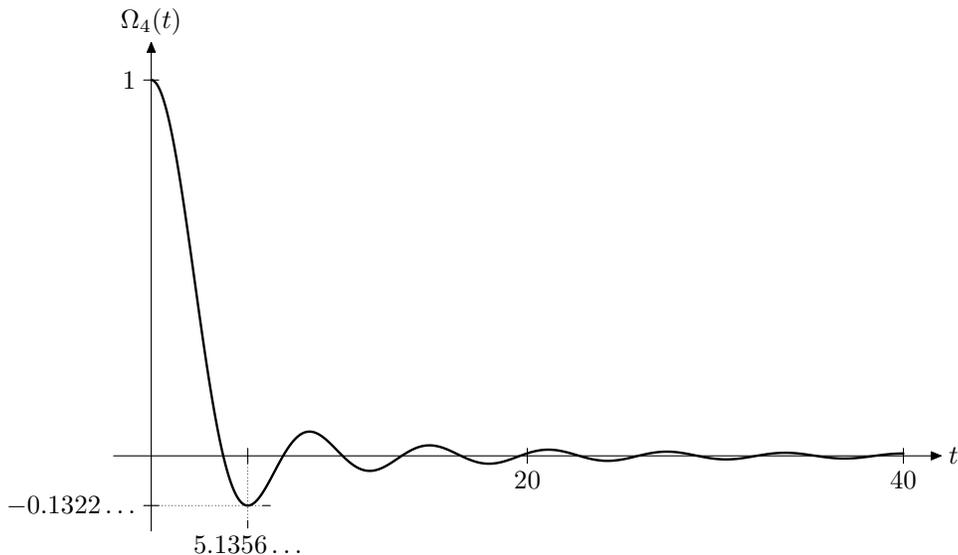}
\end{center}
\caption{Graph of the function $\Omega_4(t) = \frac{2}{t} J_1(t)$}
\label{fig:omega}
\end{figure}

\begin{theorem}
\label{th:main}
Let $d_1$, \dots,~$d_N$ be positive real numbers. Let $A \subseteq
\R^n$ be a measurable set which avoids the distances $d_1$, \dots,~$d_N$. 
Suppose there are real numbers $z_0$, $z_1$, \dots,~$z_N$ which sum
up to at least one and which satisfy
\begin{equation*}
z_0 + z_1 \Omega_n(t d_1) + z_2 \Omega_n(t d_2) + \cdots + z_N \Omega_n(t d_N) \geq 0
\end{equation*}
for all $t \geq 0$. Then, the upper density of~$A$ is at most~$z_0$.
\end{theorem}

In Section~\ref{sec:main} we provide a proof where we make
essential use of basic harmonic analysis, which we briefly recall. In
the sections that follow we apply the main theorem in a variety of
situations: sets avoiding one distance, sets avoiding two distances,
and sets avoiding many distances. For the history of these Euclidean
distance problems we refer to the surveys by Sz\'ekely~\cite{Sze2002}
and Sz\'ekely and Wormald~\cite{CFG} and the references therein.

Sets avoiding one distance have been studied by combinatorialists
because of their relation to the measurable chromatic number of the
Euclidean space. This is the minimum number of colors one needs to
color all points in~$\R^n$ so that any two points at distance~$1$
receive different colors and so that points receiving the same color
form Lebesgue measurable sets; it will be denoted
by~$\chim(\R^n)$. Since every color class of a coloring provides a
measurable set which avoids the distance~$1$, we have the inequality
\begin{equation}
\label{eq:fundamentalineq}
m_1(\R^n) \cdot \chim(\R^n) \geq 1.
\end{equation}

For the plane it is only known that $5 \leq \chim(\R^2) \leq 7$, where
the lower bound is due to Falconer~\cite{Fal} and the upper bound
comes e.g.~from a coloring one constructs using a tiling by regular
hexagons with circumradius slightly less than~$1$.  Erd\H{o}s
conjectured that $m_1(\R^2) < 1/4$ so that~\eqref{eq:fundamentalineq}
would yield an alternative proof of Falconer's result. So far the best
known results on $m_1(\R^2)$ are the lower bound $m_1(\R^2) \geq
0.2293$ by Croft~\cite{Cro} and the upper bound $m_1(\R^2) \leq 12/43
\approx 0.2790$ by Sz\'ekely~\cite{Sze1984}.  In
Section~\ref{sec:onedistance} we compute new upper bounds for
$m_{1}(\R^n)$ for dimensions $n = 2$, \dots,~$24$ based on a
strengthening of our main theorem by extra
inequalities. These new upper bounds for $m_{1}(\R^n)$ imply
by~\eqref{eq:fundamentalineq} new lower bounds for $\chim(\R^n)$ in
dimensions $3$, \dots,~$24$.

If one considers sets which avoid more than one distance one can ask
how~$N$ distances can be chosen so that the extreme density becomes as
small as possible: What is the value of $\inf\{\,m_{d_1, \ldots,
  d_N}(\R^n) : \text{$d_1$, \dots,~ $d_N > 0$}\,\}$ for fixed $N$? For
planar sets avoiding two distances Sz\'ekely \cite{Sze1984} showed
that $\inf\{\,m_{d_1, d_2}(\R^2) : d_1, d_2 > 0\,\} \leq
m_{1,\sqrt{3}}(\R^2) \leq 2/11 \approx 0.181818$. In
Section~\ref{sec:twodistances} we improve his result and show that
$\inf\{\,m_{d_1,d_2}(\R^2) : d_1,d_2 > 0\,\} \leq 0.0724046$.

Recently, Bukh~\cite{Buk}, using harmonic analysis and ideas
resembling Sz\'emere\-di's regularity lemma, showed that
$\inf\{\,m_{d_1,\ldots,d_N}(\R^n) : d_1$, \dots,~$d_N > 0\,\}$ drops
to zero exponentially in~$N$: He shows that there is a number $r$,
strictly greater than $1$, which depends only on $N$ and $n$ so that
if
\begin{eqnarray*}
\text{$d_2 / d_1 > r$, $d_3 / d_2 > r$, \dots, $d_N / d_{N-1} > r$},
\end{eqnarray*}
then $m_{d_1, \ldots, d_N}(\R^n) \leq (m_1(\R^n))^N$. This implies a
theorem of Furstenberg, Katznelson, and Weiss~\cite{FKW} that for
every subset~$A$ in the plane which has positive upper density there
is a constant~$d$ so that~$A$ does not avoid distances larger
than~$d$. Their original proof used tools from ergodic theory and
measure theory. Alternative proofs have been proposed by
Bourgain~\cite{Bou} using elementary harmonic analysis and by Falconer
and Mastrand \cite{FM} using geometric measure theory.  Bukh's result
also implies that $m_{d_1, \ldots, d_N}(\R^n)$ becomes arbitrarily
small if the distances $d_1$, $d_2$, \dots,~$d_N$ become arbitrarily
small. This is originally due to Bourgain \cite{Bou} and Falconer
\cite{Fal1986}. In Section~\ref{sec:manydistances} we give a short
proof of a variant of Bukh's result using our main theorem, where we
replace $(m_1(\R^n))^N$ by the weaker estimate $2^{-N}$. We could
improve this considerably, but we cannot get $(m_1(\R^n))^N$. Still
our estimate is strong enough to give all the implications
mentioned. Furthermore, our proof has the additional advantage that it
easily provides quantitative estimates about the spacing $r$ between
the distances.

The idea of linear programming bounds for packing problems of discrete
point sets in compact metric spaces goes back to Delsarte \cite{D} and
it has been successfully applied to a variety of situations. Cohn and
Elkies~\cite{CE} were the first who were able to set up a linear
programming bound for packing problems in non-compact spaces; by then
no less than 30 years since Delsarte's fundamental contribution had
gone by. Our main theorem can be viewed as a continuous analogue to
their linear programming bound.

\section{Proof of the main theorem}
\label{sec:main}

For the proof of our main theorem elementary notions from harmonic
analysis will be important. We recall these here. For details we refer
to, e.g., the book by Katznelson~\cite{K}.

A measurable, complex valued function $f\colon \R^n \to \C$ is called
\defi{periodic} if it is invariant under an
$n$-dimensional discrete subgroup of~$\R^n$ or, in other words, if
there is a basis $b_1$, \dots,~$b_n$ of $\R^n$ so that for all
$\alpha_1$, \dots,~$\alpha_n \in \Z$ we have $f(x +\nobreak
\sum_{i=1}^n \alpha_i b_i) = f(x)$. The set $L = \{\,\sum_{i=1}^n
\alpha_i b_i : \alpha_i \in \Z\,\}$ is called the \defi{period
  lattice} of $f$ and $L^* = \{\,u \in \R^n : \text{$x \cdot u \in \Z$
  for all $x \in L$}\,\}$ is called the \defi{dual lattice} of $L$.

The \defi{mean value} of a periodic function~$f$ is
given by
\begin{equation*}
M(f) = \lim_{T \to \infty} \frac{1}{\vol [-T,T]^n}
\int_{[-T,T]^n} f(x)\, dx.
\end{equation*}
For two periodic functions~$f$ and~$g$ we write $\langle f, g \rangle
= M(f\overline{g})$. We say that
$f$ is \defi{square-integrable} if $\langle f, f \rangle < \infty$. By
$\|f\| = \sqrt{\langle f, f \rangle}$ we denote its \defi{norm}. If
$f$ and $g$ are both square-integrable, then $\langle f, g \rangle$
exists.  For $u \in \R^n$ we define the \defi{Fourier coefficient}
$\widehat{f}(u) = \langle f, e^{i u \cdot x} \rangle$. Here, $x \cdot
y$ denotes the standard inner product on~$\R^n$.

Notice that the
support of $\widehat{f}$ is a discrete set, namely it lies in the dual
lattice of the period lattice of~$f$, scaled by~$2\pi$. If we let
$f_y(x) = f(y + x)$ for a vector $y \in \R^n$, then $\widehat{f_y}(u)
= \widehat{f}(u) e^{iu \cdot y}$. For square-integrable, periodic
functions $f$ and $g$ \defi{Parseval's formula}
\begin{equation*}
\langle f, g \rangle = \sum_{u
  \in \R^n} \widehat{f}(u) \overline{\widehat{g}(u)}
\end{equation*}
holds. By writing the latter sum we mean that we sum over the
intersection of the supports of $\widehat{f}$ and $\widehat{g}$.

\begin{proof}[Proof of Theorem~\ref{th:main}] 

  Let~$A$ be a measurable subset of~$\R^n$ that avoids distances
  $d_1$, \dots,~$d_N$. By $1_A$ we denote its characteristic function
  $1_A\colon \R^n\to\{0,1\}$ whose support is precisely $A$.  Without
  loss of generality we can assume that $1_A$ is a periodic function;
  in this case we say that $A$ is \defi{periodic}.

  Indeed, from any measurable set $A$ which avoids distances $d_1$,
  \dots,~$d_N$ we can construct a periodic set which avoids distances
  $d_1$, \dots,~$d_N$ and with upper density arbitrarily close to the
  one of $A$. To do this we intersect $A$ with a regular cube of side
  $2T$ so that $\vol(A \cap [-T,T]^n)/\vol[-T,T]^n$ is close to the
  upper density $\odelta(A)$ and so that
  $\vol([-T+d,T-d]^n)/\vol[-T,T]^n$, with $d = \max\{d_1, \ldots,
  d_N\}$, differs from $1$ only negligibly. Then we construct the new
  periodic set by tiling~$\R^n$ with copies of $A \cap [-T+d,T-d]^n$
  centered at the points of the lattice~$2T\Z^n$. Notice that, for a
  periodic set~$A$, one may replace the $\limsup$ in the definition
  of~$\odelta(A)$ by a simple limit.

  By $A - y$ we denote the translation of the set $A$ by the vector $-y
  \in \R^n$ so that its characteristic function satisfies $1_{A -
    y}(x) = 1_A(x + y) = (1_{A})_y(x)$. The following two properties are
  crucial:
\begin{eqnarray}
\label{eq:crucial1}
\langle 1_A, 1 \rangle & = & \odelta(A),\\
\label{eq:crucial2}
\langle 1_{A-y}, 1_A\rangle & = & \odelta(A \cap (A - y)),\;\;\text{for all $y \in \R^n$.}
\end{eqnarray}
In particular, we have $\langle 1_A, 1_A \rangle = \odelta(A)$ and
$\langle 1_{A-y}, 1_A \rangle = 0$ for all vectors $y$ of Euclidean
norm $d_1$, \dots,~$d_N$. Notice~$\langle 1_A, 1\rangle =
\widehat{1_A}(0)$. By applying Parseval's formula
to~\eqref{eq:crucial2}, we can express it in terms of the
Fourier coefficients of~$1_A$, thus obtaining
\begin{eqnarray*}
\widehat{1_A}(0) & = & \odelta(A),\\
\sum_{u \in \R^n} |\widehat{1_A}(u)|^2 e^{i u \cdot y} & = & \odelta(A \cap (A - y))\;\;\text{for all $y \in \R^n$}.
\end{eqnarray*}

Now we consider the function
\begin{equation}
\label{eq:autocorrelation}
\varphi(y) = \sum_{u \in \R^n}
|\widehat{1_A}(u)|^2 e^{i u \cdot y} = \odelta(A \cap (A-y)),
\end{equation}
which is called the \defi{autocorrelation function} (or \defi{two-point
correlation function}) of $1_A$. By taking spherical averages we
construct from it a radial function $f$ whose values only depend on
the norm of the vectors. In other words, we set
\begin{equation*}
f(y) = \frac{1}{\omega_n} \int_{S^{n-1}} \varphi(\|y\|\xi)\, d\omega(\xi).
\end{equation*}
Here $\omega$ denotes the standard surface measure on the unit sphere
$S^{n-1} = \{\,\xi \in \R^n : \xi \cdot \xi = 1\,\}$ and $\omega_n =
\omega(S^{n-1}) = (2\pi^{n/2})/\Gamma(n/2)$. Clearly, $f(0) =
\odelta(A)$, and $f(y) = 0$ whenever $\|y\| \in \{d_1,\ldots,d_N\}$.
Because of the formula (cf.~Schoenberg~\cite[(1.6)]{Sch}, see
\eqref{eq:omegan} for an explicit expression for $\Omega_n$)
\begin{equation*}
\frac{1}{\omega_n} \int_{S^{n-1}} e^{iu\cdot \xi}\, d\omega(\xi) = \Omega_n(\|u\|)
\end{equation*}
we can represent $f$ in the form
\begin{equation*}
f(y) = \sum_{t \geq 0} \alpha(t) \Omega_n(t \|y\|),
\end{equation*}
where~$\alpha(t)$ is the sum of $|\widehat{1_A}(u)|^2$ for vectors $u$
having norm~$t$, so the $\alpha(t)$'s are real and
nonnegative. Furthermore, $\alpha(0) = |\widehat{1_A}(0)|^2 =
\odelta(A)^2$ and $\sum_{t \geq 0} \alpha(t) = f(0) = \odelta(A)$.

So the following linear program in the variables $\alpha(t)$ gives an
upper bound for the upper density of any measurable set which avoids
the distances $d_1$, \dots,~$d_N$:
\begin{equation}
\label{eq:univariateprimal}
\begin{split}
\sup\bigl\{\,\alpha(0)\; : \; \hbox{} & \text{$\alpha(t) \geq 0$ for all $t \geq 0$},\\
& \sum_{t \geq 0} \alpha(t) = 1,\\
& \sum_{t \geq 0} \alpha(t) \Omega_n(t d_k) = 0\  \text{for $k = 1$, \dots,~$N$}\,\bigr\}.
\end{split}
\end{equation}
Above, all but a countable subset of the~$\alpha(t)$'s are zero.  Note
moreover that we used the normalization $\sum_{t \geq 0} \alpha(t) =
1$. This linear program has infinitely many variables~$\alpha(t)$ but
only $N+1$ equality constraints. A dual of it is
\begin{equation}
\label{eq:univariatedual}
\begin{split}
\inf\bigl\{\,z_0 \; : \; \hbox{} & z_0 + z_1 + z_2 + \cdots + z_N \geq 1,\\
& z_0 + z_1 \Omega_n(t d_1) + z_2 \Omega_n(t d_2) + \cdots + z_N \Omega_n(t d_N) \geq 0\\
& \qquad \text{for all $t > 0$}\, \bigr\},
\end{split}
\end{equation}
which has $N+1$ variables $z_0, z_1, z_2, \ldots, z_N$ and infinitely
many constraints.  As usual, weak duality holds between the pair of
linear programs \eqref{eq:univariateprimal} and
\eqref{eq:univariatedual}: If $\alpha(t)$ satisfies the conditions in
\eqref{eq:univariateprimal} and if $(z_0, z_1, \ldots, z_N)$ satisfies
the conditions in \eqref{eq:univariatedual}, then
\begin{equation*}
\alpha(0) \leq \sum_{t \geq 0} \alpha(t) (z_0 + z_1 \Omega_n(td_1) + z_2 \Omega_n(td_2) + \cdots + z_N \Omega_n(td_N)) = z_0,
\end{equation*}
which finishes the proof of our main theorem.
\end{proof}

\section{Sets avoiding one distance}
\label{sec:onedistance}

It is notable that the linear programming bounds for the extreme
density of sets avoiding exactly one distance allow for an analytic
optimal solution. Since this problem is scaling invariant we can
assume that we consider sets avoiding the unit distance $d_1 = 1$. Let
$j_{\alpha,k}$ be the $k$-th positive zero of the Bessel function
$J_{\alpha}$. It is known that the absolute minimum of the function
$\Omega_n$ is attained at~$j_{n/2,1}$ (see Askey, Andrews,
Roy~\cite[(4.6.2)]{AAR}, and Watson~\cite[Chapter 15, \S 31]{W}). So,
the point $(z_0, z_1)$ which is determined by the equations
\begin{eqnarray*}
z_0 + z_1 & = & 1\\
z_0 + z_1 \Omega_n(j_{n/2,1}) & = & 0
\end{eqnarray*}
provides the optimal solution for the linear program in
Theorem~\ref{th:main}. Hence,
\begin{equation}
\label{eq:solvedmystery}
z_0 = \Omega_n(j_{n/2,1})/(\Omega_n(j_{n/2,1}) - 1) \geq m_{1}(\R^n),
\end{equation}
and this gives by \eqref{eq:fundamentalineq} a lower bound for the
measurable chromatic number, namely $\chim(\R^n) \geq
1-1/\Omega_n(j_{n/2,1})$. It is interesting to notice that this lower
bound coincides with the one provided by Bachoc, Nebe, Oliveira, and
Vallentin~\cite[Corollary 8.2]{BNOV}, albeit by a shift of one
dimension. This shift of one dimension is due to the fact that Bachoc,
Nebe, Oliveira, and Vallentin~\cite{BNOV} study the problem of sets
avoiding one distance on the $(n-1)$-dimensional unit sphere $S^{n-1}
\subseteq \R^n$ and the lower bound for the measurable chromatic
number $\chim(\R^n)$ was obtained by upper bounding the density of
sets in the unit sphere which avoid the distance $d$ where $d$ goes to
zero. So, we see now that this limit process gives a lower bound for
the measurable chromatic number of $\R^{n-1}$ and not only for the one
of $\R^n$.

\subsection{Adding extra inequalities}

It is possible to strengthen the main theorem and the resulting bound
\eqref{eq:solvedmystery} by introducing extra inequalities. Consider a
regular simplex in $\R^n$ with edge length $1$ having vertices $v_1,
\ldots, v_{n+1}$. A set $A \subseteq \R^n$ which avoids the unit
distance can only contain one vertex of this regular simplex. So we
have for the autocorrelation function $\varphi$ of the characteristic
function $1_A$ defined in~\eqref{eq:autocorrelation} that
\begin{equation}
\label{eq:clique-phi}
\begin{split}
  \varphi(v_1) + \cdots + \varphi(v_{n+1}) & = \odelta(A \cap (A - v_1)) + \cdots + \odelta(A \cap (A - v_{n+1}))\\
& \leq \odelta(A)  =  \varphi(0).
\end{split}
\end{equation}

Let~$\ort(\R^n)$ be the \defi{$n$-dimensional orthogonal group}, that
is, the set of all $n \times n$ real matrices~$Z$ such that $Z^t Z =
I$. Let~$\mu$ denote the Haar measure over~$\ort(\R^n)$ normalized by
$\mu(\ort(\R^n)) = 1$. Taking spherical averages of~$\varphi$ is the
same as symmetrizing~$\varphi$ with respect to the orthogonal group,
i.e., for all~$y \in \R^n$,
\begin{equation*}
f(y) = \frac{1}{\omega_n}\int_{S^{n-1}} \varphi(\|y\|\xi)\, d\omega(\xi)
= \int_{\ort(\R^n)} \varphi(Zy)\, d\mu(Z)
\end{equation*}
(this follows, e.g., from Theorem~3.7 in the book by Mattila~\cite{Mat}).

Let~$f$ be, as above, the radial function obtained by the
symmetrization of~$\varphi$. For a nonnegative real number~$t$, we
write~$f(t)$ for the common value of~$f$ for vectors of
length~$t$. Then, by symmetrizing both sides of~\eqref{eq:clique-phi}
with respect to the orthogonal group, and since distances are
preserved by the action of~$\ort(\R^n)$, we conclude that the inequality
\begin{equation}
\label{eq:extraineq}
f(\|v_1\|) + \cdots + f(\|v_{n+1}\|) \leq 1
\end{equation}
can be used to strengthen our original linear program. Here, observe
that we already took into account the normalization~$f(0) = 1$,
introduced in~\eqref{eq:univariateprimal}.

If we center a regular simplex at the origin, the above inequality
specializes to
\begin{equation*}
(n+1) f(\sqrt{1/2-1/(2n + 2)}) \leq 1,
\end{equation*}
which gives the following strengthening of the dual formulation
\eqref{eq:univariatedual}
\begin{equation*}
\begin{split}
\inf\Big\{\,z_0 + z_c \;\; : \;\; & z_c \geq 0,\\
& z_0 + z_1 + z_c (n+1) \geq 1,\\
& z_0 + z_1 \Omega_n(t) + z_c (n+1) \Omega_n(t \sqrt{1/2-1/(2n + 2)}) \geq 0\\
& \quad  \text{for all $t \geq 0$}\,\Big\}.
\end{split}
\end{equation*}
In Table~\ref{table:thenewtable} we give the new upper bounds on
$m_1(\R^n)$ we get for $n = 4, \ldots, 24$ by solving the linear
program on a computer (we discuss numerical issues at the end of this
section) which are improvements over the values which Sz\'ekely and
Wormald give in \cite{SW}. This in turn gives new lower bounds for the
measurable chromatic number for $n = 4, \ldots, 24$.

\renewcommand{\thetable}{\arabic{section}.\arabic{table}}

\begin{table}[htb]
\begin{tabular}{r|r|r|r|r}
$n$ & best upper bound &  new upper & best lower bound & new lower\\ 
 & for $m_1(\mathbb{R}^n)$ &  bound for & for $\chi_m(\mathbb{R}^n)$ & bound for\\
 & previously known & $m_1(\mathbb{R}^n)$  & previously known &  $\chi_m(\mathbb{R}^n)$ \\
\hline
2  & 0.279069 \cite{Sze1984}   & 0.268412   & 5 \cite{Fal}     &      \\
3  & 0.187500 \cite{SW}        & 0.165609   & 6 \cite{Fal}     & 7    \\
4  & 0.128000 \cite{SW}        & 0.112937   & 8 \cite{SW}      & 9    \\
5  & 0.0953947 \cite{SW}       & 0.0752845  & 11 \cite{SW}     & 14   \\
6  & 0.0708129 \cite{SW}       & 0.0515709  & 15 \cite{SW}     & 20   \\
7  & 0.0531136 \cite{SW}       & 0.0361271  & 19 \cite{SW}     & 28   \\
8  & 0.0346096 \cite{SW}       & 0.0257971  & 30 \cite{SW}     & 39   \\
9  & 0.0288215 \cite{SW}       & 0.0187324  & 35 \cite{SW}     & 54   \\
10 & 0.0223483 \cite{SW}       & 0.0138079  & 48 \cite{BNOV}   & 73   \\
11 & 0.0178932 \cite{SW}       & 0.0103166  & 64 \cite{BNOV}   & 97   \\
12 & 0.0143759 \cite{SW}       & 0.00780322 & 85 \cite{BNOV}   & 129  \\
13 & 0.0120332 \cite{SW}       & 0.00596811 & 113 \cite{BNOV}  & 168  \\
14 & 0.00981770 \cite{SW}      & 0.00461051 & 147 \cite{BNOV}  & 217  \\
15 & 0.00841374 \cite{SW}      & 0.00359372 & 191 \cite{BNOV}  & 279  \\
16 & 0.00677838 \cite{SW}      & 0.00282332 & 248 \cite{BNOV}  & 355  \\
17 & 0.00577854 \cite{SW}      & 0.00223324 & 319 \cite{BNOV}  & 448  \\
18 & 0.00518111 \cite{SW}      & 0.00177663 & 408 \cite{BNOV}  & 563  \\
19 & 0.00380311 \cite{SW}      & 0.00141992 & 521 \cite{BNOV}  & 705  \\
20 & 0.00318213 \cite{SW}      & 0.00113876 & 662 \cite{BNOV}  & 879  \\
21 & 0.00267706 \cite{SW}      & 0.00091531 & 839 \cite{BNOV}  & 1093 \\
22 & 0.00190205 \cite{SW}      & 0.00073636 & 1060 \cite{BNOV} & 1359 \\
23 & 0.00132755 \cite{SW}      & 0.00059204 & 1336 \cite{BNOV} & 1690 \\
24 & 0.00107286 \cite{SW}      & 0.00047489 & 1679 \cite{BNOV} & 2106 \\
\end{tabular}
\\[0.3cm]
\caption{Upper bounds for $m_1(\R^n)$ and lower bounds for $\chim(\R^n)$.}
\label{table:thenewtable}
\end{table}

However, in dimension $2$ we only get an upper bound of $0.287119$. To
improve Sz\'ekely's bound of $12/43 \approx 0.279069$ in the plane, we
replace the regular triangle centered at the origin by more
triangles. We use the following three triples of squared norms
$(\|v_1\|^2, \|v_2\|^2, \|v_3\|^2)$ for \eqref{eq:extraineq}:
$(2.4,2.4,0.360314)$, $(3.1,3.1,6.524038)$ $(3.7,3.7,7.417141)$, where
the last coordinate of $(a,b,c)$ is a root of
$3(a^2+b^2+c^2+1)-(a+b+c+1)^2$. This condition assures that
the determinant of the positive semidefinite Gram matrix
\begin{equation*}
\begin{pmatrix}
a & \frac{1}{2}(a + b - 1) & \frac{1}{2}(a + c - 1)\\
\frac{1}{2}(a + b - 1) & b & \frac{1}{2}(b + c - 1)\\
\frac{1}{2}(a + c - 1) & \frac{1}{2}(b + c - 1) & c
\end{pmatrix}
\end{equation*}
of the points $v_1, v_2, v_3$ of a corresponding regular simplex
vanishes. Solving the corresponding linear program yields the new
upper bound of $0.268412$. We found the three triples by considering
all triples $(a,b,c)$ with $a,b = 0.1j$ with $j = 0, \ldots, 40$.

In dimension $3$ we use three quadruples $(\|v_1\|^2, \|v_2\|^2,
\|v_3\|^2, \|v_4\|^2)$ of squared norms for \eqref{eq:extraineq}:
$(0.3,0.4,0.4,0.417157)$, $(1.9,1.9,1.9,0.189372)$,
$(2,2,2,0.225148)$, where the last coordinate of $(a,b,c,d)$ is a root
of $3(a^2+b^2+c^2+d^2+1) - 2(ab + ac + ad + bc + bd + cd) - 2(a + b +
c + d)$. Solving the corresponding linear programming problem yields the new
upper bound of $0.165609$. We found the three quadruples by
considering all triples $(a,b,c,d)$ with $a,b,c = 0.1j$ with $j = 0,
\ldots, 40$.

\subsection{Numerical calculations}

A few technical remarks concerning the numerical calculations are in
order. For solving the linear programs we use the software {\tt
  lpsolve} \cite{lpsolve} and we generate the input using the program
{\tt GP/PARI} \cite{PARI}. We discretize the conditions of the form
\begin{eqnarray*}
z_0 + z_1 \Omega_n(t) + z_c (n+1) \Omega_n(t \sqrt{1/2-1/(2n + 2)}) \geq 0
\quad  \text{for all $t \geq 0$}
\end{eqnarray*}
by discretizing the interval $[0,20]$ into steps of size $0.0005$.

Now we demonstrate in the case $n = 4$ how we turn the numerical
calculations into a rigorous mathematical proof: The linear
program has the optimal numerical solution $z_0 = 0.0826818$, $z_1 =
0.7660402$, $z_c = 0.0302556$. A lower bound of the
minimum of the function
\begin{equation*}
z(t) = z_0 + z_1 \Omega_4(t) + 5z_c \Omega_4(\sqrt{2/5}t)
\end{equation*}
in $t \in [0,20]$ is $-0.00000006$. The function $z(t)$ is positive
for $t \geq 20$ because there $\Omega_4(t) \geq -0.02$ and
$\Omega_4(\sqrt{2/5}t) \geq -0.04$ holds. Thus by adding $0.00000006$
to $z_0$ we make sure that the new function $z(t)$ is
nonnegative. This only slightly effects the value of the bound.

\section{Planar sets avoiding two distances}
\label{sec:twodistances}

In this section we quickly report on the problem of finding the smallest
extreme density a measurable set in the plane can have which avoids
exactly two distances, i.e., $\inf\{\,m_{d_1, d_2}(\R^2) : d_1, d_2 >
0\,\}$.  Sz\'ekely \cite{Sze1984} showed that this number is at most
$2/11$ by giving an upper bound for $m_{1,\sqrt{3}}(\R^2)$. By solving
the corresponding linear program on the computer we improve his bound
to $m_{1,\sqrt{3}}(\R^2) \leq 0.170213$. By adjusting the distances we
can improve this further: $m_{1,j_{1,2}/j_{1,1}}(\R^2) \leq 0.141577$
where $j_{1,1}$ and $j_{1,2}$ are the first two positive zeros of the
Bessel function $J_1$.

By combining Bukh's result, which we explained in the introduction,
with our new bound on $m_1(\R^2)$ from the previous section we can
improve on this even further: 
\begin{equation*}
\inf\{\,m_{d_1, d_2}(\R^2) : d_1, d_2 > 0\,\} \leq
\left(m_1(\R^2)\right)^2 \leq 0.072046.
\end{equation*}

\section{Sets avoiding many distances}
\label{sec:manydistances}

In this section we give a proof of a variant of Bukh's
result~\cite[Theorem~1]{Buk} about densities of sets avoiding many
distances. His proof is based on a so-called zooming out lemma which
resembles Szemer\'edi's regularity lemma for dense graphs, whereas our
proof is an easy consequence of Theorem~\ref{th:main} and simple
properties of the function~$\Omega_n$.

\begin{theorem}
\label{th:bukh}
For every positive integer~$N$ there is a number $r = r(N)$ strictly
greater than~$1$ such that for distances~$d_1$, \dots,~$d_N$ with
\begin{equation}
\label{eq:bukh}
\text{$d_2 / d_1 > r$, $d_3 / d_2 > r$, \dots, $d_N / d_{N-1} > r$}
\end{equation}
we have $m_{d_1, \ldots, d_N}(\R^n) \leq 2^{-N}$.
\end{theorem}

In the proof of Theorem~\ref{th:bukh} some facts about the
function~$\Omega_n$ will be useful. First, we have
\begin{equation}
\label{eq:jbounds}
|J_0(t)| \leq 1,\quad\text{and}\quad |J_\alpha(t)| \leq 1 / \sqrt{2}
\quad\text{for all $\alpha > 0$ and $t \geq 0$}
\end{equation}
(cf.~(4.9.13) in~Andrews, Askey, and~Roy~\cite{AAR}). From this, it
follows at once that $\lim_{t\to\infty} \Omega_n(t) = 0$ for~$n >
2$. For~$n = 2$ the same follows, e.g., from the asymptotic expansion
for~$J_\alpha$ (cf.~(4.8.5) in~Andrews, Askey, and~Roy~\cite{AAR}). 

Moreover,
\begin{equation}
\label{eq:omegamin}
\Omega_n(t) \geq -1/2,\quad\text{for all $n \geq 2$ and $t \geq 0$}.
\end{equation}
To see this, set $\alpha = (n-2)/2$. It is known that for~$\alpha >
-1/2$ we have
\[
J_{\alpha-1}(t) + J_{\alpha+1}(t) = \frac{2\alpha}{t} J_\alpha(t)
\]
(cf. Andrews, Askey, and~Roy~\cite[(4.6.5)]{AAR}). Combining this
identity with~\eqref{eq:omegan} we obtain
\begin{equation}
\label{eq:omegarec}
\begin{split}
\Omega_n(t) & = \Gamma(\alpha + 1) \Bigl(\frac{2}{t}\Bigr)^\alpha
\frac{t}{2\alpha} (J_{\alpha-1}(t) + J_{\alpha+1}(t))\\
 &= \Omega_{n-2}(t) + \Gamma(\alpha)
 \Bigl(\frac{2}{t}\Bigr)^{\alpha-1} J_{\alpha+1}(t).
\end{split}
\end{equation}

Now, recall that the global minimum of~$\Omega_n$ is attained
at~$j_{\alpha+1,1}$, the first positive zero of~$J_{\alpha+1}$
(cf.~Section~\ref{sec:onedistance}). From~\eqref{eq:omegarec}, we have
that $\Omega_n(j_{\alpha+1, 1}) = \Omega_{n-2}(j_{\alpha+1,1})$. It
follows that the minimum of~$\Omega_n$ is at least the minimum
of~$\Omega_{n-2}$. To finish the proof of~\eqref{eq:omegamin} we have
to check~$\Omega_2$ and~$\Omega_3$, what can be easily accomplished.

\begin{proof}[Proof of Theorem~\ref{th:bukh}]
Given~$N > 0$, set $\varepsilon = 1/(N 2^{N+1})$.
Since $\Omega_n(0) = 1$ and since~$\Omega_n$ is continuous, there is a
number $t_0 > 0$ such that $\Omega_n(t) > 1 - \varepsilon$ for $t \leq
t_0$. Likewise, since $\lim_{t\to\infty} \Omega_n(t) = 0$, there is a
number $t_1 > t_0$ such that $|\Omega_n(t)| < \varepsilon$ for $t \geq
t_1$.

Set $r = r(N) = t_1 / t_0$ and let distances $d_1$, \dots,~$d_N$ be
given such that~\eqref{eq:bukh} is satisfied. With this we claim that,
for $1 \leq j \leq N$,
\begin{equation*}
\sum_{i=j}^N \frac{1}{2^{N-i+1}}\cdot\Omega_n(t d_i) \geq
-\frac{1}{2^{N-j+2}} - (N - j)\varepsilon.
\end{equation*}

Before we prove the claim, we show how to apply it.  By taking $j = 1$
in the claim, and since by our choice of~$\varepsilon$ we have
$-(N-1)\varepsilon \geq -1 / 2^{N+1}$, it follows that
\begin{eqnarray*}
\sum_{i=1}^N \frac{1}{2^{N-i+1}} \cdot \Omega_n(t d_i) \geq -
\frac{1}{2^N}.
\end{eqnarray*}
Now we may set $z_0 = 1 / 2^N$ and $z_i = 1 / 2^{N-i+1}$ for $i = 1$,
\dots,~$N$ and apply Theorem~\ref{th:main}, proving our result.

To finish, we prove the claim by induction. For~$j = N$, the statement
follows immediately from~\eqref{eq:omegamin}. Now, suppose the
statement is true for some $1 < j \leq N$. We show that the statement
is also true for~$j - 1$ by distinguishing two cases.

First, for $t \leq t_0 / d_{j-1}$, we have from the choice of~$t_0$
that $\Omega_n(t d_{j-1}) > 1 - \varepsilon$. Using this and the
induction hypothesis, we then have that
\begin{equation*}
\begin{split}
\sum_{i=j-1}^N \frac{1}{2^{N-i+1}}\cdot\Omega_n(td_i) &=
\frac{1}{2^{N-j+2}}\cdot\Omega_n(t d_{j-1}) + 
\sum_{i=j}^N \frac{1}{2^{N-i+1}}\cdot\Omega_n(td_i)\\
&\geq \frac{1 - \varepsilon}{2^{N-j+2}} - \frac{1}{2^{N-j+2}} -
(N-j)\varepsilon\\
&\geq -\frac{1}{2^{N-j+3}} - (N-j+1)\varepsilon.
\end{split}
\end{equation*}

Now suppose $t \geq t_0 / d_{j-1}$. Observe that, for $j \leq i \leq
N$, we have $t d_i \geq t_0 d_i / d_{j-1} \geq t_0 r = t_1$, hence
$|\Omega_n(td_i)| < \varepsilon$. So, by using~\eqref{eq:omegamin}, we
have
\begin{equation*}
\begin{split}
\sum_{i=j-1}^N \frac{1}{2^{N-i+1}}\cdot\Omega_n(td_i) &=
\frac{1}{2^{N-j+2}}\cdot\Omega_n(t d_{j-1}) + 
\sum_{i=j}^N \frac{1}{2^{N-i+1}}\cdot\Omega_n(td_i)\\
&\geq -\frac{1}{2^{N-j+3}} - (N - j + 1) \varepsilon,
\end{split}
\end{equation*}
finishing the proof of the claim.
\end{proof}

\section*{Acknowledgements}

We thank Christine Bachoc for helpful discussions, Noam D.~Elkies for
enlightening comments, and Boris Bukh and Lex Schrijver for their
valuable suggestions on our manuscript.

\end{document}